\documentclass[12pt]{amsart}
\usepackage{amsmath, amssymb, comment, float, graphicx, setspace}
 \pdfoutput=1
\usepackage[T1]{fontenc}
\usepackage{parskip}

\def\theenumi{\arabic{enumi}}

\def\theenumii{\alph{enumii}}
\def\p@enumii{\theenumii.}

\def\theenumiii{\arabic{enumiii}}
\def\p@enumiii{(\theenumi)(\theenumii)}

\def\p@enumiv{\p@enumiii.\theenumiii}

\newtheorem{theorem}{Theorem}

\newtheorem{conjecture}[theorem]{Conjecture}
\newtheorem{corollary}[theorem]{Corollary}

\newtheorem{lemma}[theorem]{Lemma}

\newtheorem{definition}{Definition}

\newtheorem*{lemma*}{Lemma}

\theoremstyle{remark}

\newcommand{\Z}{\mathbb{Z}}

\newcommand{\B}{\mathfrak{B}}

\title{Zero Divisor Graphs of Quotient Rings}
\author{Rachael Alvir}

\begin{document}

\maketitle
 
 \begin{abstract}
The \emph{compressed zero-divisor graph} $\Gamma_C(R)$ associated with a commutative ring $R$ has vertex set equal to the set of equivalence classes $\{ [r] \mid r \in Z(R), r \neq 0 \}$ where $r \sim s$ whenever $ann(r) = ann(s)$. Distinct classes $[r],[s]$ are adjacent in $\Gamma_C(R)$ if and only if $xy = 0$ for all $x \in [r], y \in [s]$. In this paper, we explore the compressed zero-divisor graph associated with quotient rings of unique factorization domains. Specifically, we prove several theorems which exhibit a method of constructing $\Gamma(R)$ for when one quotients out by a principal ideal, and prove sufficient conditions for when two such compressed graphs are graph-isomorphic. We show these conditions are not necessary unless one alters the definition of the compressed graph to admit looped vertices, and conjecture necessary and sufficient conditions for two compressed graphs with loops to be isomorphic when considering any quotient ring of a unique factorization domain.
 \end{abstract}

\section{Introduction}

In \emph{Coloring of Commutative Rings} \cite{beck}, I. Beck introduces the \emph{zero-divisor graph} $\Gamma(R)$  associated with the zero-divisor set of a commutative ring, whose vertex set is the set of zero-divisors. Two distinct zero-divisors $x,y$ are adjacent in $\Gamma(R)$ if and only if $xy = 0$. The zero-divisor graph establishes a "connection between graph theory and commutative ring theory which hopefully will turn out to mutually beneficial for those two branches of mathematics."  M.  Axtell and J. Stickles \cite{axtell}  remark that "in general, the set of zero-divisors lacks algebraic structure," suggesting that turning to the zero-divisor graph may both reveal both ring-theoretical properties and impose a graph-theoretical  structure. Beck's hopes have certainly been met; his now classical paper motivated an explosion of research in this and similar associated graphs in the past decade. His definition has since been modified to emphasize the fundamental structure of the zero-divisor set: Anderson's definition given in \cite{anderson}, which excludes the vertex $0$ from the graph, is now considered standard.

In 2009, S. Spiroff and C. Wickham formalized the \emph{compressed zero-divisor graph} $\Gamma_C(R)$, which has vertex set equal to the set of equivalence classes $\{ [r] \mid r \in Z(R), r \neq 0 \}$ where $r \sim s$ whenever $ann(r) = ann(s)$. Distinct classes $[r],[s]$ are adjacent in $\Gamma_C(R)$ if and only if $xy = 0$ for all $x \in [r], y \in [s]$. It is proved in \cite{mulay} that $xy = 0$ for all $x \in [r], y \in [s]$ if and only if $rs = 0$. In other words, multiplication of equivalence classes is well-defined. This graph "represents a more succinct description of the zero-divisor activity" \cite{spiroff}.  Spiroff's compressed graph was inspired by S. Mulay's work in \cite{mulay} and has also been called the Mulay graph or the graph of equivalence classes of zero-divisors. The compressed graph has many advantages over the traditional zero-divisor graph, as it is often finite in cases where the zero-divisor set is infinite, and reveals associated primes.

The purpose of this paper is to study the compressed zero-divisor graph, and discover in what way it better describes the core behavior of the zero-divisor set. This question becomes particularly tractable in unique factorization domains, building off intuition built in $\Z / \langle n \rangle$. Our main theorem reveals that the overall structure of the zero-divisor activity for quotient rings of unique factorization domains is determined by a finite set of key elements, which we call the \emph{zero-divisor basis} for the ring. As a result of this observation, we are able to state sufficient conditions for two compressed graphs of such rings to be graph-isomorphic, and conjecture necessity when the compressed graph admits looped vertices. In the final section, we conjecture that these results may be extended to any quotient ring of a unique factorization domain.

\section{Background}

One motivation for the use of the compressed zero-divisor graph is the observation that new zero-divisors may be trivially found; every multiple of a zero-divisor is also a zero-divisor. Moreover, these multiples often behave in exactly the same way as their parent zero-divisor: they annihilate the same elements. This phenomena is expressed in the following lemma.

\begin{lemma}
Let $R$ be a commutative ring, $z \in Z(R)$, $n \not \in Z(R)$. Then $[nz] = [z]$. 
\label{nozd}
\end{lemma}
\begin{proof}
($\subseteq$) Let $x \in ann(nz)$. Then $(nz)x = 0 \Rightarrow n(zx) = 0$. Since $n \not \in Z(R)$, it follows by negating the definition of zero-divisor that $zx = 0$. Therefore $x \in ann(z)$. \\
($\supseteq$) Let $x \in ann(z)$. Then $zx = 0 \Rightarrow (nz)x = 0 \Rightarrow x \in ann(nz)$.
\end{proof}

The above is useful when considering the zero-divisor graph and thus appears in the literature, although often in its weaker form where $n$ is a unit. This observation leads us to consider that some zero-divisors do not determine the behavior of the zero-divisor set, and have their behavior determined by other "important" zero-divisors. Naturally, one wishes to identify which zero-divisors actually determine the zero-divisor graph's structure.
To this end, we define two zero-divisors $r,s$ to be equivalent if $ann(r) = ann(s)$, and denote the equivalence class of $r$ with $[r]$. Because copycat zero-divisors add unnecessary visual information to $\Gamma(R)$, we wish to associate a graph with this equivalence relation. The compressed zero-divisor graph then becomes the new object of study.

Are there any other situations, other than that expressed in lemma \ref{nozd}, in which the annihilators of two elements are equal? This is addressed by the following lemma: 

\begin{lemma} 
Let $S$ be a commutative ring, $c \in S$, $I$ an ideal of $S$. Then $c + I \in Z(S/I)$ if and only if $\exists d \in S-I$ such that $cd + I = I$.
\label{zd}
\end{lemma}
\begin{proof}
\begin{align*}
c+I \in Z (S / I) & \iff \exists d \in S \text{ s.t. } cd+I = I \text{ and } d+I \neq I \\
& \iff \exists d \in S \text{ s.t. } cd+I = I \text{ and } d \not \in I
\end{align*}
\end{proof}

We may restate this lemma so that it explicitly relates to the zero-divisor graph. We follow standard notation in graph theory: given a graph $G$, its vertex set is denoted $V(G)$, and edge set $E(G)$.

\begin{corollary}
Let $S$ be a commutative ring, $I$ an ideal of $S$. Then $\exists c, d \in S-I$ such that $cd + I = I$, $c+I \neq d + I$ if and only if 
\begin{enumerate}
\item $c + I \in V(\Gamma(S/I))$, and
\item $e = \{c+I,d+I\} \in E(\Gamma(S/I))$. 
\end{enumerate}
\end{corollary}
\begin{proof} The extra condition that $c \in S-I$ ensures that $0$ is not a vertex of $\Gamma(S/I)$. The condition that $c+I \neq d + I$ ensures that edges exist only between distinct vertices.
\end{proof}

Using this simple intuition as a springboard, we proceeed to consider the consequences of this relationship between the ideal and the zero-divisor set of the quotient ring.

\section{Results}

Corollary 3 indicates that in order to investigate the zero-divisor set of a quotient ring, one must look at the factors of the elements in the kernel the homomorphism. Not every factor, however, will do. If $I$ is the principal ideal $\langle n \rangle$, we will wish to restrict our attention only to the factors of $n$. We therefore utilize the following definition. 

\begin{definition}
Let $D$ be a unique factorization domain and $\langle n \rangle$ a principal ideal of $D$. A \textbf{Zero-Divisor Basis} $\mathfrak{B}$ for $D / \langle n \rangle$ is a set of all nontrivial divisors of $n$ such that no two distinct elements in the basis are associates. In symbols,
\begin{itemize}
\item $d \in \B \Rightarrow d \mid n, d$ is not a unit, $d$ is not an associate of $n$.
\item $d_1, d_2 \in \B \Rightarrow d_1 = d_2$ or $d_1, d_2$ are not associates.
\end{itemize}
\end{definition}

It is our claim that such factors are the only ones which determine the behavior of the zero-divisor set, i.e., every zero-divisor has the same annihilator as the image of some irreducible divisor of $n$.

\begin{lemma}
Let $D$ be a UFD, $I$ an ideal of $D$, and $\phi$ the canonical homomorphism from $D$ onto $D / I$. Let $a,n \in D$. Then $[\phi(\gcd(a,n))]$ is well-defined.
\end{lemma}
\begin{proof}
If $I$ is trivial, then there is only one annihilator class in $D/I$ and we are done. If not, then since $\phi$ is onto, the image of a unit under $\phi$ is a unit. But a unit is never a zero-divisor. By Lemma \ref{nozd}, images of associates therefore share equivalence classes. However, every $\gcd(a,n)$ is unique up to associates, so $[\phi(\gcd(a,n))]$ is well-defined.
\end{proof}

\begin{theorem}
Let $D$ be a UFD, $\langle n \rangle$ a principal ideal of $D$, and $\phi$ the canonical homomorphism from $D$ onto $D / \langle n \rangle$. Then $\forall a \in D, [\phi(a)] = [\phi(\gcd(a,n))]$.
\end{theorem}

\begin{proof}
If $\langle n \rangle$ is maximal or trivial, there is only one annihilator class in $D/\langle n \rangle$ and we are done. If not, then $n$ is nonzero and not a unit. Let $n = \prod p_i^{s_i}$ be a factorization of $n$ into irreducibles. Fix $a \in D$.

If $a = 0$, then $[\phi(\gcd(0,n))] =  [\phi(n)] = [0]$ and we are done. 

Suppose $a \neq 0$. We may write $a = y \prod p_i^{k_i}$ where $\gcd(y,n) = 1$. Then by Lemma $\ref{zd}$ and since $y \neq 0$, $\phi(y)$ is not a zero-divisor, so  $[\phi(a)] = [\phi(\prod p_i^{k_i})].$ But 
$[\phi(\prod p_i^{k_i})] = [\phi(\gcd(a,n))]$ if and only if $\forall x \in D, x \prod p_i^{k_i} \in \langle n \rangle  \iff x \gcd(a,n) \in \langle n \rangle$. It therefore suffices to show that the equivalence $\forall x \in D, n \mid x \prod p_i^{k_i}  \iff n \mid x \gcd(a,n)$ holds. \\

($\Leftarrow$) Let $x \in D$, and suppose $n \mid x \gcd(a,n)$. Since $\gcd(y,n) = 1$, then $\gcd(a,n) = \gcd(\prod p_i^{k_i}, n)$, which implies that $\gcd(a,n) \mid \prod p_i^{k_i}$. We conclude that $n \mid x \prod p_i^{k_i}$. \\
($\Rightarrow$)
Let $x \in D$, and suppose $n \mid x \prod p_i^{k_i}$. Then 
$$\prod p_i^{min(k_i, s_i)} \prod p_i^{s_i-min(k_i,s_i)} \big| x \prod p_i^{min(k_i,s_i)} \prod p_i^{k_i - min(k_i,s_i)}.$$
Since the cancellation property holds in UFD's, 
$$\prod p_i^{s_i-min(k_i,s_i)} \big| x \prod p_i^{k_i - min(k_i,s_i)}.$$
However, by definition of minimum, if $k_i - min(k_i, s_i) \neq 0$, then  $s_i - min(k_i,s_i) = 0$. In other words, $\gcd( \prod p_i^{k_i - min(k_i,s_i)}, p_i^{s_i - min(k_i,s_i)}) = 1.$ 
By Euclid's Lemma, we conclude that
$$\prod p_i^{s_i-min(k_i,s_i)} \big| x,$$
which implies that
$$\prod p_i^{min(k_i, s_i)} \prod p_i^{s_i-min(k_i,s_i)} \big| x \prod p_i^{min(k_i, s_i)}.$$
Therefore
$n \mid x \prod p_i^{min(k_i,s_i)}$, which is what we desired to show.
\end{proof}

\begin{theorem}
Let $D$ be a unique factorization domain, and $\langle n \rangle$ a nontrivial principal ideal of $D$. Let $\phi$ be the canonical homomorphism from $D$ onto $D / \langle n \rangle$. 

Each vertex of $\Gamma_C(D / \langle n \rangle)$ is the equivalence class $[\phi(d)]$ of exactly one $d$ in the zero-divisor basis $\B$ for $D / \langle n \rangle$. An edge exists between two distinct nodes $[\phi(d_1)], [\phi(d_2)]$ if and only if $d_1 d_2 \in \langle n \rangle$. 
\label{zdbasis}
\end{theorem}

\begin{proof}
If $\langle n \rangle$ is maximal, then $\B$ is empty and so also is $\Gamma_C(D/\langle n \rangle)$. Suppose $\langle n \rangle$ is not maximal. As $\langle n \rangle$ is neither maximal nor trivial, let $n = \prod p_i^{s_i}$.

It has already been shown that $[\phi(a)] = [\phi(\gcd(a,n))]$ for all $a \in D$. Thus every equivalence class is the equivalence class of a divisor of $n$. Since the vertices of $\Gamma_C(D / \langle n \rangle)$ do not contain the classes $[0]$ or $[1]$, every vertex of $\Gamma_C(D / \langle n \rangle)$ is the equivalence class of $\phi(d)$ for some $d$ in $\B$.

We show next that this is true for exactly one $d \in \B$. Suppose that $d_1, d_2$ are distinct elements in the zero-divisor basis for $D/\langle n \rangle$. Because the zero-divisor basis is unique up to associates, we may write
$d_1 = u \prod p_i^{t_i}, d_2 = v \prod p_i^{k_i}$ where $u,v$ are units. Since $d_1 \neq d_2$ and by definition of zero-divisor basis they are not associates, some $t_j \neq k_j$. Without loss of generality, let $k_j > t_j$. Consider $d_3= \prod p_i^{s_i-k_i}$. Then clearly $\phi(d_3)$ annihilates $\phi(d_2)$. However, $d_1d_3 = u \prod p_i^{t_i + (s_i-k_i)} = u \prod p_i^{s_i+t_i-k_i}$. But $k_j > t_j$ implies that $s_j + t_j-k_j < s_j$, so $\phi(d_1) \phi(d_3) \neq 0$. We conclude that $ann(\phi(d_1)) \neq ann(\phi(d_2))$.

That an edge exists between distinct vertices $[\phi(d_1)], [\phi(d_2)]$ if and only if $d_1, d_2 \in \langle n \rangle$ follows immediately from the definition of compressed zero-divisor graph.
\end{proof}

\begin{corollary}
The compressed graph of $D / \langle n \rangle$ is always finite.
\end{corollary}
\begin{proof}
If $\langle n \rangle$ is maximal or trivial, we are done because the compressed graph is empty. If not, then because $n$ may be expressed as a finite product of irreducibles, the zero-divisor basis of $D / \langle n \rangle$ is always finite.
\end{proof}

\begin{theorem}
Let $D_1, D_2$ be two unique factorization domains, and let $n_1 \in D_1$ and $n_2 \in D_2$ be non-zero and not units. Let
$n_1 = \prod_{i=1}^{A} p_i^{s_i}, n_2 = \prod_{i=1}^{B} q_i^{t_i}$ be factorizations of $n_1, n_2$ into products of irreducibles.
If $A = B$ and each $s_i = t_i$, then $\Gamma_C(D_1 / \langle n_1 \rangle) \cong \Gamma_C(D_2 / \langle n_2 \rangle).$
\end{theorem}

\begin{proof}
A zero-divisor basis of $D_1 / \langle n_1 \rangle$ is 
$$ \B_1 = \{ d = \prod_{i=1}^A p_i^{v_i} \mid d \neq n, d \neq 1, 0 \leq v_i \leq s_i \},$$ and a basis of $D_1 / \langle n_2 \rangle$ is given by
$$ \B_2 = \{ d = \prod_{i=1}^B p_i^{w_i} \mid d \neq n, d \neq 1, 0 \leq w_i \leq s_i \}.$$ 

Consider the map $\psi$ (whose domain does not contain the equivalence classes of 0 or 1) defined by
$$\psi \left( \left[ \prod p_i^{z_i} + \langle n_1 \rangle \right] \right) = \left[ \prod q_i^{z_i} + \langle n_2 \rangle \right].$$

By Theorem \ref{zdbasis}, $\psi$ is a map between the vertex sets of $\Gamma_C(D_1 / \langle n_1 \rangle)$ and $\Gamma_C(D_2 / \langle n_2 \rangle)$. Since $A = B$, each $s_i = t_i$, and again by Theorem \ref{zdbasis} each vertex corresponds to exactly one element in the zero-divisor basis, we conclude that $\psi$ is a one-to-one correspondence.

Suppose $[\prod p_i^{f_i}], [\prod p_i^{g_i}]$ are adjacent in $\Gamma_C(D_1 / \langle n_1 \rangle)$.
This happens if and only if 
\begin{align*}
& \prod p_i^{f_i} \prod p_i^{g_i} + \langle n_1 \rangle = \langle n_1 \rangle \\
\iff &\prod p_i^{f_i+g_i} \in \langle n_1 \rangle \\
\iff &\forall i, f_i + g_i \geq s_i = t_i \\
\iff &\prod q_i^{f_i+g_i} \in \langle n_2 \rangle \\
\iff &\prod q_i^{f_i} \prod q_i^{g_i} + \langle n_2 \rangle = \langle n_2 \rangle \\
\iff &\psi \left( \left[\prod p_i^{f_i} \right] \right), \psi \left( \left[ \prod p_i^{g_i} \right] \right) \text{ are adjacent in } \Gamma_C(D_2 / \langle n_2 \rangle).
\end{align*}
Therefore $\psi$ is an isomorphism.
\end{proof}

\begin{corollary}
Let $D$ be a UFD and $\langle n \rangle$ a principal ideal of $D$. Then $\Gamma_C(D/\langle n \rangle)$ is isomorphic to $\Gamma_C(\Z / \langle m \rangle)$ for some $m \in \Z$. 
\end{corollary}

It is perhaps appropriate we mention that requiring $\Gamma(R)$ (and consequently $\Gamma_C(R)$) be simple seems to have caused some disagreement among the zero-divisor graph community. Axtell's definition of zero-divisor graph in \cite{axtell} does not require that edges exist only between distinct vertices and refers explicitly to "loops" in the graph, despite the fact that he claims Anderson's definition is standard. Anderson himself must remind his reader that "the zero-divisor graph does not detect the nilpotent elements" \cite{anderson}. For Redmond, Anderson's definition posed potential problems, requiring "some conditions that make the nilpotent elements easier to locate" \cite{redmond}.

In this paper also, we advocate changing the standard definitions of $\Gamma(R)$ and $\Gamma_C(R)$ so that looped vertices are allowed. First, only a graph with loops models \emph{all} zero-divisor activity in the ring. Additionally, although the former theorem holds for either definition of $\Gamma_C(R)$, its conditions are not necessary unless we stipulate that $\Gamma_C(R)$ contain loops. As a counterexample, note that in the unlooped case, $\Gamma_C(\Z / \langle p^3 \rangle) \cong \Gamma_C(\Z / \langle pq \rangle)$. We believe that sufficiency is held when $\Gamma_C(R)$ admits loops, and that this sufficiency need only be shown for $\Z / \langle m \rangle$ by the above corollary. 

For certain quotient rings, the next theorem gives necessary and sufficient conditions for compressed graph isomorphism regardless of whether or not we admit loops in $\Gamma_C(R)$.

\begin{theorem} 
Let $D_1, D_2$ be two UFD's. Let $p, p_1, \ldots, p_A \in D_1$, $q, q_1, \ldots q_B \in D_2$ be irreducibles, where no two of the $p_i,q_i$'s are associates respectively. Then
$\Gamma_C(D_1 / \langle p^s \rangle) \cong \Gamma_C(D_2 / \langle q^t \rangle)$ if and only if $s = t$, and $\Gamma_C(D_1 / \langle \prod_{i=1}^{A}p_i \rangle) \cong \Gamma_C(D_2 / \langle \prod_{i=1}^{B} q_i \rangle)$ if and only if $A = B$.
\end{theorem}

\begin{proof}
Sufficiency has already been shown. We prove necessity by contraposition.

In the first case, suppose $s \neq t$. Then it is clear that $p^s$, $q^t$ have different numbers of non-trivial divisors up to associates, so the zero-divisor bases for the two rings will have different cardinalities. Therefore the vertex sets of $\Gamma_C(D_1 / \langle p^s \rangle)$, $\Gamma_C(D_2 / \langle q^t \rangle)$ cannot be put into a one-to-one correspondence, and the graphs cannot be isomorphic.

In the second case, suppose $A \neq B$. Similarly, it is clear that 
the vertex sets of $$\Gamma_C(D_1 / \langle \prod^{A}p_i \rangle), \Gamma_C(D_2 / \langle \prod^{B} q_i \rangle)$$ have different cardinalities.
\end{proof}

These graph-isomorphism results generalize those in \cite{endean}, in which is is proved that $\Gamma(\Z_p[x]/\langle x^n \rangle) \cong \Gamma(\Z_{p^n})$.

\section{Conjectures}
We would like to discuss when our results may be extended to $\Gamma(R)$ itself. The following conjectures sugges when this might be possible.

\begin{conjecture} Let $R_1,R_2$ be two commutative rings. Then $\Gamma(R_1) \cong \Gamma(R_2)$ if and only if 
\begin{enumerate}
\item $\Gamma_C(R_1) \cong \Gamma_C(R_2)$, and
\item $\vert R_1 - Z(R_1) \vert = \vert R_2 - Z(R_2) \vert$.
\end{enumerate}
\end{conjecture}

We would also like to extend our results to when one quotients a UFD by any ideal $I$.
Note first that every ideal of a ring may be written as a union of principal ideals. The trivial union $I = \cup_{x \in I} \langle x \rangle$ suffices to show that this is true; equality holds because ideals are closed under multipliciation by the elements of a ring by definition. A more general definition of zero-divisor basis follows.

\begin{definition}
Let $D$ be a unique factorization domain, and let $I = \cup_{\alpha \in \Lambda} \langle n_\alpha \rangle$ be an expression of $I$ as a union of principal ideals. A \textbf{zero-divisior basis} $\B$ of $D / I$ is the set of all divisors $d$ of some $n_\alpha$, $d \not \in I$, such that no two distinct elements of $\B$ are associates.
\end{definition}

If $I$ is expressed minimally, we may be able to replace the stipulation that $d \not \in I$ with the requirement that the divisor $d$ merely be nontrivial.

\begin{conjecture}
Let $D$ be a UFD, $I = \cup_{\alpha \in \Lambda} \langle n_\alpha \rangle$ be an ideal of $D$ expressed as a minimal union, and $\phi$ the canonical homomorphism from $D$ onto $D / I $.
Then $\forall a \in D, [\phi(a)] = [\phi(\gcd \{a, n_{\alpha}  \text{ for all }\alpha \in \Lambda \}]$.
\end{conjecture}

\begin{conjecture}
Let $D$ be a unique factorization domain, and $I$ an ideal of $D$ where 
$$ I = \cup_{\alpha \in \Lambda} \langle \prod_{k=1}^{Q_\alpha} p_{\alpha_k}^{s_{\alpha_k}} \rangle$$
is a minimal expression of $I$ as a union of principal ideals. The distinct nodes of the compressed zero-divisor graph for $D/I$ are given by the equivalence classes of elements of the form $\phi(\prod d_i)$ where each $d_i$ is an element of the zero-divisor basis. An edge exists between two nodes $\phi(\prod d_i), \phi(\prod b_i)$ if and only if $n_\alpha \mid \prod d_ib_i$ for some $\alpha \in \Lambda$.
\end{conjecture}

\begin{conjecture}Let $D_1$, $D_2$ be unique factorization domains and $I_1, I_2$ be any two ideals of $D_1, D_2$ respectively which are neither maximal nor trivial.
Suppose $$ I_1 = \cup_{\alpha \in \Lambda} \langle \prod_{k=1}^{Q_\alpha} p_{\alpha_k}^{s_{\alpha_k}} \rangle.$$
Then $\Gamma_C(D_1/I_1) \cong \Gamma_C(D_2/I_2)$ if we may write
$$ I_2 = \cup_{\alpha \in \Lambda} \langle \prod_{k=1}^{Q_\alpha} q_{\alpha_k}^{s_{\alpha_k}} \rangle.$$
\end{conjecture}

\end{document}